\documentclass[12pt]{amsart}
    \usepackage[OT2,T1]{fontenc}
    \usepackage{amsmath,amsfonts,amsthm,amssymb,mathrsfs,
    amsxtra,amscd,latexsym,color,mathdots,verbatim}
    \usepackage[hmargin=2cm,vmargin=3cm]{geometry}
    \usepackage{microtype}
    \DeclareSymbolFont{cyrletters}{OT2}{wncyr}{m}{n}
    \DeclareMathSymbol{\Sha}{\mathalpha}{cyrletters}{"58}

     \newtheorem{thm}{Theorem}[section]
     
     \newtheorem{cor}[thm]{Corollary}
     
     \newtheorem{lem}[thm]{Lemma}

     \newtheorem{dfn}[thm]{Definition}

     \newtheorem{rmk}[thm]{Remark}
     
     \theoremstyle{definition}
     \theoremstyle{remark}
     \numberwithin{equation}{section}

    \newcommand{\sm}{\left(\begin{smallmatrix}}
    \newcommand{\esm}{\end{smallmatrix}\right)}
    \newcommand{\mat}{\left(\begin{matrix}}
    \newcommand{\emat}{\end{matrix}\right)}

    \newcommand{\mbf}{\mathbf}

    \def\CC{\mathbb{C}}
    \def\HH{\mathbb{H}}

    \def\GL{\mathrm{GL}}
    \def\Mp{\mathrm{Mp}}
    \def\SL{\mathrm{SL}}

\begin{document}

    \title{Non-Vanishing of L-Functions of Vector-Valued Modular Forms}

   \author{Subong Lim and Wissam Raji}
\date{}

\maketitle

\begin{abstract}
Kohnen proved a non-vanishing result for $L$-functions associated to Hecke eigenforms of integral weights on the full group. 
In this paper, we show a non-vanishing result for the averages of $L$-functions associated with the orthogonal basis of the space of cusp forms of vector-valued modular forms of weight $k\in \frac12 \mathbb{Z}$ on the full group. We also show the existence of at least one basis element whose L-function does not vanish under certain conditions.
As an application, we generalize the result of Kohnen to $\Gamma_0(N)$ and prove the analogous result for Jacobi forms. 
\end{abstract} %%%%%%%%%%%%%%%%%%%%%%%%%%%%%%%%%%%%%%%%%%%%%%%%%%%%%%%%%%%%%%%%%%%%%%%%

    %\thanks{}

   % \maketitle
   
 \section{Introduction}  
 Vector-Valued modular forms have played a crucial role in the theory of modular forms.  In particular, Selberg used these forms to give an estimation for the Fourier coefficients of the classical modular forms \cite{S}. Moreover, vector-valued modular forms arise naturally in the theory of Jacobi forms, Siegel modular forms, and Moonshine. Some applications of vector-valued modular forms stand out in high-energy physics by mainly providing a method of differential equations in order to construct the modular multiplets, and also revealing the simple structure of the modular invariant mass models \cite{DL}. Other applications concerning vector-valued modular forms of half-integer weight seem to provide a simple solution to the Riemann-Hilbert problem for representations of the modular group \cite{BG}.
 
\par In \cite {K},\cite{KM2}, and \cite{KM}, Knopp and/or Mason gave a systematic development of the theory of vector-valued modular forms where they introduced the foundation of the space of these forms mainly through the introduction of vector-valued Poincar\'e series and vector-valued Eisenstein series leading to a better understanding of the space of vector-valued modular forms.  More recently, several algorithms for computing Fourier coefficients of vector-valued modular forms were determined in connection to Weil representations due to their importance in the Moonshine applications \cite{MR}.

\par On the other hand, $L$-functions of vector-valued modular forms play important role in the above-mentioned computations as well so it is natural to study them.  
In this paper, we show a non-vanishing result for averages of $L$-functions associated with vector-valued modular forms.  To illustrate, we let $\{f_{k,1}, \ldots, f_{k,d_k}\}$ be an orthogonal basis of the space of cusp forms $S_{k,\chi,\rho}$ with Fourier coefficients $b_{k,l,j}(n)$, where $\chi$ is a multiplier system of weight $k\in\frac12 \mathbb{Z}$ on $\mathrm{SL}_2(\mathbb{Z})$ and
$\rho:\SL_2(\mathbb{Z})\to \GL_m(\CC)$ is an $m$-dimensional unitary complex representation.
We let $t_0\in\mathbb{R}, \epsilon>0$, and $1\leq i\leq m$. 
  Then, there exists a constant $C(t_0, \epsilon, i)>0$ such that for $k>C(t_0, \epsilon, i)$ the function 
  \[
  \sum_{l=1}^{d_k} \frac{<L^*(f_{k,l}, s),\mbf{e}_i>}{(f_{k,l}, f_{k,l})} b_{k,l,i}(n_{i,0})
  \]
  does not vanish at any point $s=\sigma + it_0$ with $\frac{k-1}{2}<\sigma < \frac k2 - \epsilon$, where $<L^*(f_{k,l}, s),\mbf{e}_i>$ denotes the $i$th component of $L^*(f_{k,l}, s)$.
By using the integral weight case, we generalize a result of Kohnen in \cite{Koh} to $\Gamma_0(N)$ in Section \ref{Section5}. 
On the other hand, by using the half-integral weight case, we  prove the analogous result for Jacobi forms in Section \ref{Section6}.

 %  Furthermore,  we assume that $\rho\left(\sm 0&-1\\1&n\esm\right)$ is a diagonal matrix whose $(j,j)$ entry is $e^{2\pi i\kappa_j n}$ for each integer $n$.
 
 \section{Preliminaries}
  Let $k\in \frac12\mathbb{Z}$ and $\chi$ a unitary multiplier system of weight $k$ on $\Gamma$, i.e. $\chi:\mathrm{SL}_2(\mathbb{Z})\to\mathbb{C}$ satisfies the following conditions:
  \begin{enumerate}
  \item $|\chi(\gamma)| = 1$ for all $\gamma\in \mathrm{SL}_2(\mathbb{Z})$.
  \item $\chi$ satisfies the consistency condition
  \[
  \chi(\gamma_3) (c_3\tau + d_3)^k = \chi(\gamma_1)\chi(\gamma_2) (c_1\gamma_2\tau + d_1)^k (c_2\tau+d_2)^k,
  \]
  where $\gamma_3 = \gamma_1\gamma_2$ and $\gamma_i = \sm a_i&b_i\\c_i&d_i\esm\in \mathrm{SL}_2(\mathbb{Z})$ for $i=1,2$, and $3$.
  \end{enumerate}
  Let $m$ be a positive integer and $\rho:\mathrm{SL}_2(\mathbb{Z})\to \mathrm{GL}_m(\mathbb{C})$ a $m$-dimensional unitary complex representation.
  We assume that $\rho(-I)$ is the identity matrix, where $I$ denotes the identity matrix. 
  Let $\{\mbf{e}_1, \ldots, \mbf{e}_m\}$ denote the standard basis of $\mathbb{C}^m$. 
  For a vector-valued function $f = \sum_{j=1}^m f_j \mbf{e}_j$ on $\mathbb{H}$ and $\gamma\in \Gamma$, define a slash operator by
  \[
  (f|_{k,\chi,\rho}\gamma)(\tau):= \chi^{-1}(\gamma)(c\tau+d)^{-k} \rho^{-1}(\gamma) f(\gamma \tau).
  \]

  \begin{dfn} \label{defofvvm}
%   Let $k\in \frac12\mathbb{Z}$.
  A vector-valued modular form of weight $k$ and multiplier system $\chi$ with respect to $\rho$ on $\mathrm{SL}_2(\mathbb{Z})$ is a sum $f = \sum_{j=1}^m f_j \mbf{e}_j$ of functions holomorphic in  $\mathbb{H}$ satisfying the following conditions:
  \begin{enumerate}
  \item $f|_{k,\chi,\rho}\gamma = f$ for all $\gamma\in \mathrm{SL}_2(\mathbb{Z})$.
  \item For each $1\leq j\leq m$, each function $f_j$ has a Fourier expansion of the form
  \[
  f_i(\tau) = \sum_{n+\kappa_j\geq 0} a_j(n)e^{2\pi i(n+\kappa_j)\tau}.
  \]
  Here and throughout the paper, $\kappa_j$ is a certain positive number with $0\leq \kappa_j<1$. 
  \end{enumerate}
  We write $M_{k,\chi,\rho}$ for the space of vector-valued modular forms of weight $k$ and multiplier system $\chi$ with respect to $\rho$ on $\SL_2(\mathbb{Z})$.
   There is a subspace $S_{k,\chi,\rho}$ of vector-valued cusp forms for which we require that each $a_j(n) = 0$ when $n+\kappa_j$ is non-positive.
  \end{dfn}

From the condition (2) in Definition \ref{defofvvm}, we see that   $\chi\left(\sm 1&1\\0&1\esm\right) \rho\left(\sm 1&1\\0&1\esm\right)$ is a diagonal matrix whose $(j,j)$ entry is $e^{2\pi i\kappa_j}$.
%Given a representation $\rho$ of $\Gamma$, a certain (nonnegative) constant $\alpha$, which depends only on $\rho$, was introduced in \cite{KM2}.
 If $f\in S_{k, \chi, \rho}$ is a vector-valued cusp form,  then $a_j(n) = O(n^{k/2})$ for every $1\leq j\leq m$, as $n\to\infty$ by the same argument for classical modular forms (for example, see \cite[Section 1]{KM}).
For a vector-valued cusp form $f(z) = \sum_{j=1}^m \sum_{n+\kappa_j>0}a_j(n)e^{2\pi i(n+\kappa_j)z}\mbf{e}_j$ we define the $L$-series
\[L(f,s) = \sum_{j=1}^m  \sum_{n+\kappa_j>0} \frac{a_j(n)}{(n+\kappa_j)^s}\mbf{e}_j.\]
This series converges absolutely for  $\mathrm{Re}(s)\gg0$.

The following theorem for vector-valued modular forms follows from the same argument used for classical modular forms. 

\begin{thm} \label{cusp} 
 Let $k\in \frac12\mathbb{Z}$.
If  $f\in S_{k,\rho}$ is a vector-valued cusp form, then 
\[\frac{\Gamma(s)}{(2\pi)^s}L(f,s) = \int_{0}^\infty f(iy)y^s\frac{dy}{y}.\]
Furthermore, $L(f,s)$ has an analytic continuation  and functional equation
\[ L^*(f,s) =i^k \chi(S) \rho(S)L^*(f, k-s),\]
where $$L^*(f,s) =\frac{\Gamma(s)}{(2\pi)^s}L(f,s)$$ and $S = \sm 0&-1\\1&0\esm$. 
\end{thm}

\section{The Construction of the Kernel Function}
In what follows, we define the kernel function $R_{k,s,l}$ which will play a crucial role in determining the Fourier coefficients of the orthogonal basis of the space of vector-valued cusp forms using Petersson's scalar product. Moreover, we determine the Fourier coefficients of this kernel function using the Lipshitz summation formula.
\par Let $l$ be an integer with $1\leq l\leq m$.
Define 
\[
p_{s,l} (\tau) := \tau^{-s}\mbf{e}_l.
\]
For $\tau\in\mathbb{H}$ and $s\in\mathbb{C}$ with $1<\mathrm{Re}(s)<k-1$, we define
\[
R_{k, s, l} := \gamma_k(s) \sum_{\gamma \in \mathrm{SL}_2(\mathbb{Z})}   p_{s,l} |_{k,\chi,\rho}\gamma,
\]
where $\gamma_k(s) := \frac12 e^{\pi is/2} \Gamma(s) \Gamma(k-s)$.

We write $<\cdot, \cdot>$ for the standard scalar product on $\mathbb{C}^m$, i.e. 
\[
<\sum_{j=1}^m \lambda_j \mbf{e}_j, \sum_{j=1}^m \mu_j \mbf{e}_j> = \sum_{j=1}^m \lambda_j \overline{\mu_j}.
\]
Then, for $f, g\in M_{k,\rho}$, we define the Petersson scalar product of $f$ and $g$ by
\[
(f,g) := \int_{\mathcal{F}} <f(\tau), g(\tau)> v^k \frac{dudv}{v^2}
\]
if the integral converges, where $\mathcal{F}$ is the standard fundamental domain for the action of $\mathrm{SL}_2(\mathbb{Z})$ on $\mathbb{H}$.

\begin{lem} \label{Petersson}
 Let $k\in \frac12\mathbb{Z}$ with $k>2$, and let $s\in\mathbb{C}$ with $1<\mathrm{Re}(s)<k-1$.
\begin{enumerate}
\item
The series $R_{k,s,l}$ converges absolutely uniformly whenever $\tau = u+iv$ satisfies $v\geq \epsilon, u\leq 1/\epsilon$ for a given $\epsilon>0$, and $s$ varies over a compact set. 

\item The series $R_{k,s,l}$ is a vector-valued cusp form in $S_{k,\chi,\rho}$. 

\item For $f\in S_{k,\chi,\rho}$, we have
\[
(f, R_{k,\bar{s}, l}) = c_k <L^*(f,s), \mbf{e}_l>,
\]
where $c_k:= \frac{(-1)^{k/2} \pi (k-2)!}{2^{k-2}}$.

\end{enumerate}
\end{lem}

\begin{proof}
For the first part, note that for each $1\leq j\leq m$, we have
\[
(R_{k,s,l})_j (\tau) = \gamma_k(s) \sum_{\sm a&b\\c&d\esm\in \SL_2(\mathbb{Z})} \chi^{-1}\left(\sm a&b\\c&d\esm\right)   \rho^{-1}\left(\sm a&b\\c&d\esm\right)_{j,l} (c\tau+d)^{-k}\left( \frac{a\tau+b}{c\tau+d}\right)^{-s},
\]
where $\rho^{-1}\left(\sm a&b\\c&d\esm\right)_{j,l}$ denotes the $(j,l)$-th entry of the matrix $\rho^{-1}\left(\sm a&b\\c&d\esm\right)$.
Then, we have
\[
 \sum_{\sm a&b\\c&d\esm\in \SL_2(\mathbb{Z})} \left| \chi^{-1}\left(\sm a&b\\c&d\esm\right)\rho^{-1}\left(\sm a&b\\c&d\esm\right)_{j,l} (c\tau+d)^{-k}\left( \frac{a\tau+b}{c\tau+d}\right)^{-s}\right| \leq \sum_{\sm a&b\\c&d\esm \in \mathrm{SL}_2(\mathbb{Z})} \left| (c\tau+d)^{-k} \left( \frac{a\tau+b}{c\tau+d}\right)^{-s}\right|
 \]
since $\rho$ is a unitary representation.
It is known that the series 
\[
 \sum_{\sm a&b\\c&d\esm \in \SL_2(\mathbb{Z})} (c\tau+d)^{-k} \left( \frac{a\tau+b}{c\tau+d}\right)^{-s}
\]
converges absolutely uniformly whenever  $\tau = u+iv$ satisfies $v\geq \epsilon, u\leq 1/\epsilon$ for a given $\epsilon>0$, and $s$ varies over a compact set (see \cite[Section 4]{Koh}).
Therefore, the series $R_{k,s,l}$ converges absolutely uniformly whenever $\tau = u+iv$ satisfies $v\geq \epsilon, u\leq 1/\epsilon$ for a given $\epsilon>0$, and $s$ varies over a compact set. 
The second part follows readily from that.

For the last part, we follow the argument in \cite[Lemma 1]{Koh}.
It is enough to consider the case when $1<\mathrm{Re}(s) < \frac{k-1}{2}$.
Note that for each $(c,d)\in\mathbb{Z}^2$ with $(c,d)=1$, we can find $(a, b)\in\mathbb{Z}^2$ such that $a d - b c = 1$. 
Then, we see that $R_{k,s,l}(\tau)$ is equal to 
\[
\gamma_k(s) \sum_{\substack{(c,d)\in\mathbb{Z}^2 \\ (c,d)=1}} \sum_{n\in\mathbb{Z}} (c\tau+d)^{-k} \left(\frac{a\tau + b}{c\tau+d} + n\right)^{-s} 
\chi^{-1}\left( \sm a&b\\c&d\esm \right)\chi^{-1}\left(\sm 1&n\\0&1\esm \right)
\rho^{-1}\left( \sm a&b\\c&d\esm \right)\rho^{-1}\left(\sm 1&n\\0&1\esm \right)  \mbf{e}_l,
\]
where for each coprime pair $(c,d)\in\mathbb{Z}^2$, one chooses a fixed pair $(a,b)\in \mathbb{Z}^2$ such that $ad-bc=1$.
Therefore, we have
\[
(R_{k,s,l})_j(\tau) = \gamma_k(s) \sum_{\substack{(c,d)\in\mathbb{Z}^2 \\ (c,d)=1}} \sum_{n\in\mathbb{Z}} (c\tau+d)^{-k} \left(\frac{a\tau + b}{c\tau+d} + n\right)^{-s} e^{-2\pi in \kappa_l} \chi^{-1}\left(\sm a&b\\c&d\esm\right) \rho^{-1}\left( \sm a&b\\c&d\esm \right)_{j,l}. 
\]
Next, we will use the Lipschitz summation formula \cite{Lip}: 
For $0\leq a\leq 1, \mathrm{Re}(s)>1$ and $\tau\in\mathbb{H}$, we have
\begin{equation} \label{Lipschitz}
\frac{\Gamma(s)}{(-2\pi i)^s} \sum_{k\in\mathbb{Z}} \frac{e^{2\pi iak}}{(k+\tau)^s} = \sum_{n=1}^\infty \frac{e^{2\pi i\tau(n-a)}}{(n-a)^{1-s}}.
\end{equation}
Therefore, we have
\begin{eqnarray*}
(R_{k,s,l})_j(\tau) &=&  \frac{1}{2} (2\pi)^s \Gamma(k-s) \sum_{n+\kappa_l>0} (n+\kappa_l)^{s-1} \\
&&\times\sum_{\substack{ (c,d)\in\mathbb{Z}^2 \\ (c,d)=1}}\chi^{-1}\left(\sm a&b\\c&d\esm\right) 
\rho^{-1}\left( \sm a&b\\c&d\esm \right)_{j,l} (c\tau+d)^{-k} e^{2\pi i(n+\kappa_l)(a \tau+b)/(c\tau+d)}.
\end{eqnarray*}
From this, we have
\begin{eqnarray*}
R_{k,s,l}(\tau) = (2\pi)^s \Gamma(k-s) \sum_{n+\kappa_l>0} (n+\kappa_l)^{s-1}P_{k,n,l}(\tau),
\end{eqnarray*}
where 
\[
P_{k,n,l}(\tau) = \frac12 \sum_{\substack{ (c,d)\in\mathbb{Z}^2 \\ (c,d)=1}}\chi^{-1}\left(\sm a&b\\c&d\esm\right) \rho^{-1}\left( \sm a&b\\c&d\esm \right) (c\tau+d)^{-k} e^{2\pi i(n+\kappa_l)(a \tau+b)/(c\tau+d)} \mbf{e}_l
\]
is a vector-valued Poincar\'e series.
By following the argument as in \cite[Theorem 5.3]{KM}, we have
\[
(f, P_{k,n,l}) = a_{n}(l) \frac{\Gamma(k-1)}{(4\pi (n+\kappa_l))^{k-1}}.
\]
Therefore, we see that
\begin{eqnarray*}
(f, R_{k, \bar{s}, l}) &=& c_k (-1)^{k/2} (2\pi)^{-(k-s)} \Gamma(k-s) <L(f, k-s), \mbf{e}_l>\\
&=& c_k <L^*(f,s), \mbf{e}_l>.
\end{eqnarray*}
\end{proof}

  We now compute the Fourier expansion of $R_{k,s,l}$ by following a similar argument as in \cite[Lemma 2]{Koh}.

\begin{lem} \label{Fourier}
 Let $k\in \frac12\mathbb{Z}$ with $k>2$.
The function $R_{k,s,l}$ has the Fourier expansion
\[
R_{k,s,l}(\tau) = \sum_{j=1}^m \sum_{n+\kappa_j>0} r_{k,s,l,j}(n)e^{2\pi i(n+\kappa_j)\tau},
\]
where $r_{k,s,l,j}(n)$ is given by
\begin{eqnarray*}
r_{k,s,l,j}(n) &=& \delta_{l,j}  (2\pi)^s \Gamma(k-s)(n+\kappa_l)^{s-1} + \chi^{-1}(S)   \rho^{-1}\left( S \right)_{j,l} (-1)^{k/2}(2\pi)^{k-s}\Gamma(s) (n+\kappa_j)^{k-s-1}\\
&&+\frac{ (-1)^{k/2}}{2} (2\pi)^k (n+\kappa_j)^{k-1} \frac{\Gamma(s)\Gamma(k-s)}{\Gamma(k)}\sum_{\substack{(c,d)\in\mathbb{Z}^2 \\ (c,d)=1, ac>0}}    c^{-k} \left( \frac ca \right)^{s}\\
&&\quad\times \bigg(e^{2\pi i(n+\kappa_j)d/c} e^{\pi is}\chi^{-1}\left( \sm a&b\\c&d\esm\right) \rho^{-1}\left( \sm a&b\\c&d\esm\right)_{j,l}\ _1F_1(s, k; -2\pi in/(ac))\\
&&\quad \quad + e^{-2\pi i(n+\kappa_j)d/c} e^{-\pi is} \chi^{-1}\left( \sm -a&b\\c&-d\esm\right) \rho^{-1} \left( \sm -a&b\\c&-d\esm\right)_{j,l}\ _1F_1(s, k; 2\pi in/(ac)) \bigg),
\end{eqnarray*}
where $_1 F_1(\alpha, \beta;z)$ is Kummer's degenerate hypergeometric function.
\end{lem}

\begin{proof}
  First, we consider the contribution of the terms where $ac = 0$. 
  The contribution of the terms $\pm \sm 1&n\\0&1\esm$ 
  \[
  2\gamma_k(s) \sum_{n\in\mathbb{Z}} (z+n)^{-s}e^{-2\pi n\kappa_l} \mbf{e}_l
  \]
  can be written as follows by the Lipschitz summation formula in (\ref{Lipschitz}) 
  \[
  (2\pi)^s \Gamma(k-s) \sum_{n+\kappa_l>0} (n+\kappa_l)^{s-1} e^{2\pi i(n+\kappa_l) z} \mbf{e}_l.
  \]
  Note that $  \sm 0&-1\\1&n \esm = \sm 0&-1\\1&0\esm \sm 1&n\\0&1\esm$.
  Therefore, the contribution of the terms $\pm \sm 0&-1\\1&n\esm$ is equal to 
  \begin{eqnarray} \label{anonzero}
 \nonumber && 2\gamma_k(s) \sum_{n\in\mathbb{Z}} (-z)^{-s} (z+n)^{s-k} 
   \chi^{-1}\left(\sm 1&n\\0&1\esm\right) \chi^{-1} \left( \sm 0&-1\\1&0\esm\right)
  \rho^{-1}\left(\sm 1&n\\0&1\esm\right) \rho^{-1} \left( \sm 0&-1\\1&0\esm\right)\mbf{e}_l\\
  &&= 2\gamma_k(s) (-z)^{-s} \sum_{j=1}^m \chi^{-1} \left( \sm 0&-1\\1&0\esm\right)\rho^{-1}\left( \sm 0&-1\\1&0\esm \right)_{j,l} \sum_{n\in\mathbb{Z}}  e^{-2\pi in\kappa_j}(z+n)^{s-k}\mbf{e}_j.
  \end{eqnarray}
  By the similar computation as in the case of the terms  $\pm \sm 1&n\\0&1\esm$, we see that (\ref{anonzero}) is equal to 
  \[
  (-1)^{k/2}(2\pi)^{k-s}\Gamma(s)\sum_{j=1}^m \chi^{-1} \left( \sm 0&-1\\1&0\esm\right) \rho^{-1}\left( \sm 0&-1\\1&0\esm \right)_{j,l}  \sum_{n+\kappa_j>0} (n+\kappa_j)^{k-s-1}e^{2\pi i(n+\kappa_j)z} \mbf{e}_j.
  \]
  
  The contribution of the terms with $ac\neq0$ at the $j$-th component is given by
  \begin{eqnarray} \label{contribution}
 &&\gamma_{k}(s)\int_{iC}^{iC+1} \sum_{\substack{(c,d)\in\mathbb{Z}^2 \\ (c,d)=1, ac\neq0}} (cz+d)^{-k} \left(\frac{az+d}{cz+d}\right)^{-s}
 \chi^{-1}\left( \sm a&b\\c&d\esm\right)
 \rho^{-1}\left( \sm a&b\\c&d\esm\right)_{j,l}  e^{-2\pi i(n+\kappa_j)z} dz\\
\nonumber   &&=\gamma_k(s)\int_{iC}^{iC+1}   \sum_{m\in\mathbb{Z}} \sum_{\substack{(c,d)\in\mathbb{Z}^2 \\ (c,d)=1, ac\neq0}} (c(z+m)+d)^{-k} \left(\frac{a(z+m)+b}{c(z+m)+d}\right)^{-s}  \\
\nonumber &&\qquad \qquad \qquad \times e^{-2\pi im\kappa_j} \chi^{-1}\left( \sm a&b\\c&d\esm\right) \rho^{-1}\left( \sm a&b\\c&d\esm\right)_{j,l}e^{-2\pi i(n+\kappa_j)z} dz
  \end{eqnarray}
  for any fixed positive real number $C$. 
  By the change of variables $z\mapsto z+m (m\in\mathbb{Z})$, we see that (\ref{contribution}) is equal to 
  \begin{equation} \label{computation1}
\gamma_k(s)\sum_{\substack{(c,d)\in\mathbb{Z}^2 \\ (c,d)=1, ac\neq0}} \chi^{-1}\left( \sm a&b\\c&d\esm\right) \rho^{-1}\left( \sm a&b\\c&d\esm\right)_{j,l} \int_{iC-\infty}^{iC+\infty} (cz+d)^{-k} \left(\frac{az+b}{cz+d}\right)^{-s} e^{-2\pi i(n+\kappa_j)z} dz.
  \end{equation}
  By the change of variables $z\mapsto z-d/c$, we see that (\ref{computation1}) is equal to 
  \begin{equation} \label{computation2}
  \gamma_k(s) \sum_{\substack{(c,d)\in\mathbb{Z}^2 \\ (c,d)=1, ac\neq0}} c^{-k} e^{2\pi i(n+\kappa_j)d/c} \chi^{-1}\left( \sm a&b\\c&d\esm\right) \rho^{-1}\left( \sm a&b\\c&d\esm\right)_{j,l}\int_{iC-\infty}^{iC+\infty}  z^{-k} \left(-\frac{1}{c^2 z}+\frac{a}{c}\right)^{-s}e^{-2\pi i(n+\kappa_j)z} dz.
  \end{equation}
  
  If $ac>0$, then we have
  \[
  z^{-s} \left(-\frac{1}{c^2 z}+\frac{a}{c}\right)^{-s}  = \left( -\frac{1}{c^2} + \frac{a}{c} z\right)^{-s}.
  \]
  Therefore, by the change of variable $z\mapsto (c/a)it$, we see that the integral in  (\ref{computation2}) is equal to
\begin{equation} \label{computation3}
(-1)^{k/2} 2\pi \left(\frac{c}{a} \right)^{-k+s+1} \frac{1}{2\pi i}\int_{C-i\infty}^{C+i\infty} t^{-k+s}\left(t+\frac{i}{c^2}\right)^{-s} e^{2 \pi (n+\kappa_j)(c/a)t} dt.
\end{equation}
Note that for $\mathrm{Re}(\mu), \mathrm{Re}(\nu)>0, p\in\mathbb{C}$, we have
\[
\frac{1}{2\pi i}\int_{C-i\infty}^{C+i\infty} (t+\alpha)^{-\mu} (t+\beta)^{-\nu} e^{pt} dt = \frac{1}{\Gamma(\mu+\nu)} p^{\mu+\nu-1}e^{-\beta p}\ _1F_1(\mu, \mu+\nu; (\beta-\alpha)p)
\]
(see \cite{Erd}).
 Therefore, (\ref{computation3}) can be written as
 \[
 (-1)^{k/2} \frac{(2\pi)^k}{\Gamma(k)} (n+\kappa_j)^{k-1} \left(\frac {c}{a} \right)^s\ _1F_1(s, k; -2\pi in/(ac)).
 \]
 From this, we see that the contribution of the terms with $ac>0$ at the $j$-th component is equal to 
 \begin{eqnarray*}
&&\frac{ (-1)^{k/2}}{2} (2\pi)^k (n+\kappa_j)^{k-1} \frac{\Gamma(s)\Gamma(k-s)}{\Gamma(k)} \\
&&\times\sum_{\substack{(c,d)\in\mathbb{Z}^2 \\ (c,d)=1, ac>0}} c^{-k} \left( \frac ca \right)^{s} e^{2\pi i(n+\kappa_j)d/c} e^{\pi is}\chi^{-1}\left( \sm a&b\\c&d\esm\right) \rho^{-1}\left( \sm a&b\\c&d\esm\right)_{j,l}\ _1F_1(s, k; -2\pi in/(ac)).
 \end{eqnarray*}
 The contribution of the terms with $ac<0$ at the $j$-th component is obtained by the same argument if we replace $(a,c)$ by $(-a,c)$.
 \end{proof}
 
  \section{The Main Result}

  In this section, we give the main result where we determine the non-vanishing of the averages of L-functions associated with the orthogonal basis of the space of cusp forms.  We also show the existence of at least one basis element whose L-function does not vanish under certain conditions.
    Let 
    \begin{equation*}
  n_{j,0} := 
  \begin{cases}
  1 & \text{if $\kappa_j = 0$},\\
  0 & \text{if $\kappa_j\neq0$}.
  \end{cases}
  \end{equation*}

  \begin{thm} \label{nonvanishing}
 % Let $k>2+2\alpha$. 
  Let $k\in \frac12\mathbb{Z}$ with $k>2$.
   Let $\{f_{k,1}, \ldots, f_{k,d_k}\}$ be an orthogonal basis of $S_{k,\chi,\rho}$ with Fourier expansions
 \[
 f_{k,l}(z) = \sum_{j=1}^m \sum_{n+\kappa_j>0} b_{k,l,j}(n)e^{2\pi i(n+\kappa_j)z}\ (1\leq l\leq d_k).
 \]
  Let $t_0\in\mathbb{R}, \epsilon>0$, and $1\leq j\leq m$. 
  Then, there exists a constant $C(t_0, \epsilon, j)>0$ such that for $k>C(t_0, \epsilon, j)$ the function 
  \[
  \sum_{l=1}^{d_k} \frac{<L^*(f_{k,l}, s),\mbf{e}_i>}{(f_{k,l}, f_{k,l})} b_{k,l,j}(n_{j,0})
  \]
  does not vanish at any point $s=\sigma + it_0$ with $\frac{k-1}{2}<\sigma < \frac k2 - \epsilon$.
  %and $\frac k2 + \epsilon < \sigma < \frac{k+1}2$.
  \end{thm}
  
  \begin{rmk}
  Note that
  \[
  \sum_{l=1}^{d_k} \frac{<L^*(f_{k,l}, s),\mbf{e}_i>}{(f_{k,l}, f_{k,l})} b_{k,l,j}(n_{j,0}) 
  \]
  is the $n_{j,0}$th Fourier coefficient of $R_{k, \bar{s}, j}$. 
  Therefore, it is equal to a nonzero constant multiple of 
  \[
 (R_{k, \bar{s}, j}, P_{k,n_{j,0}, j}).
 \]
 Therefore, Theorem \ref{nonvanishing} implies the nonvanishing of $<L^*(P_{k,n_{j,0},j}, s), \mbf{e}_j>$. 
  \end{rmk}

  \begin{proof}
%By the functional equation in Theorem \ref{cusp}, it suffices to consider the left half of the critical strip.
 By Lemma \ref{Petersson}, we have
 \begin{eqnarray*}
 R_{k,\bar{s}, j} = \sum_{l=1}^{d_k} \frac{(f_{k,l}, R_{k,\bar{s}, j})}{(f_{k,l}, f_{k,l})}  f_{k,l} = c_k\sum_{l=1}^{d_k} \frac{<L^*(f_{k,l}, s), \mbf{e}_j>}{(f_{k,l}, f_{k,l})} f_{k,l}.  
 \end{eqnarray*}
  If we take the first Fourier coefficients of both sides at the $j$th component, then by Lemma \ref{Fourier} we have
\begin{eqnarray}\label{compare}
\nonumber &&c_k \sum_{l=1}^{d_k} \frac{<L^*(f_{k,l}, s), \mbf{e}_j>}{(f_{k,l}, f_{k,l})} b_{k,l,j}(n_{j,0})\\
\nonumber &&=  (2\pi)^s \Gamma(k-s)\tilde{\kappa}_j^{s-1} + \chi^{-1}\left( \sm 0&-1\\1&0\esm \right) \rho^{-1}\left( \sm 0&-1\\1&0\esm \right)_{j,j} (-1)^{k/2}(2\pi)^{k-s}\Gamma(s) \tilde{\kappa}_j^{k-s-1}\\
 &&\quad +\frac{ (-1)^{k/2}}{2} (2\pi)^k \tilde{\kappa}_j^{k-1} \frac{\Gamma(s)\Gamma(k-s)}{\Gamma(k)}\sum_{\substack{(c,d)\in\mathbb{Z}^2 \\ (c,d)=1, ac>0}}    c^{-k} \left( \frac ca \right)^{s}\\
\nonumber &&\quad \quad\times \bigg(e^{2\pi i(n+\kappa_j)d/c} e^{\pi is} \chi^{-1}\left( \sm a&b\\c&d\esm\right) \rho^{-1}\left( \sm a&b\\c&d\esm\right)_{j,j}\ _1F_1(s, k; -2\pi in/(ac))\\
\nonumber &&\quad \quad  \quad + e^{-2\pi i(n+\kappa_j)d/c} e^{-\pi is} \chi^{-1}\left( \sm -a&b\\c&-d\esm\right) \rho^{-1}\left( \sm -a&b\\c&-d\esm\right)_{j,j}\ _1F_1(s, k; 2\pi in/(ac)) \bigg),
  \end{eqnarray}
  where 
  \begin{equation*}
  \tilde{\kappa}_j := 
  \begin{cases}
  1 & \text{if $\kappa_j = 0$},\\
  \kappa_i & \text{if $\kappa_j\neq0$}.
  \end{cases}
  \end{equation*}
  
  Suppose that 
  \[
  c_k \sum_{l=1}^{d_k} \frac{<L^*(f_{k,l}, s), \mbf{e}_j>}{(f_{k,l}, f_{k,l})} b_{k,l,j}(n_{j,0}) = 0.
  \]
  If we divide (\ref{compare}) by $(2\pi)^s \Gamma(k-s)\kappa_j^{s-1}$, we have
  \begin{eqnarray} \label{compare2}
 \nonumber -1& =&\rho^{-1}\left( \sm 0&-1\\1&0\esm \right)_{j,j}  (-1)^{k/2}(2\pi)^{k-2s} \frac{\Gamma(s)}{\Gamma(k-s)}\tilde{\kappa}_j^{k-2s} + \frac{(-1)^{k/2} (2\pi)^{k-s} \tilde{\kappa}_j^{k-s}}{2\Gamma(k-s)}\sum_{\substack{(c,d)\in\mathbb{Z}^2 \\ (c,d)=1, ac>0}}    c^{-k} \left( \frac ca \right)^{s}\\
  && \times \bigg( e^{2\pi i(n+\kappa_j)d/c} e^{\pi is} \chi^{-1}\left( \sm a&b\\c&d\esm\right) \rho^{-1}\left( \sm a&b\\c&d\esm\right)_{j,j}\ _1f_1(s, k; -2\pi in/(ac))\\
 \nonumber &&\qquad 
 + e^{-2\pi i(n+\kappa_j)d/c} e^{-\pi is}  \chi^{-1}\left( \sm -a&b\\c&-d\esm\right) \rho^{-1}\left( \sm -a&b\\c&-d\esm\right)_{j,j}\ _1f_1(s, k; 2\pi in/(ac)) \bigg),
  \end{eqnarray}
  where 
  \[
  \ _1f_1(\alpha, \beta;z) := \frac{\Gamma(\alpha)\Gamma(\beta-\alpha)}{\Gamma(\beta)}\ _1F_1(\alpha, \beta;z).
  \]
    Let $s = \frac k2 - \delta - it_0$, where $\epsilon < \delta<\frac12$. 
  Then, we have
  \begin{eqnarray} \label{compare3}
   \nonumber -1& &= \chi^{-1}\left( \sm 0&-1\\1&0\esm \right)  \rho^{-1}\left( \sm 0&-1\\1&0\esm \right)_{j,j} (-1)^{k/2}(2\pi \tilde{\kappa}_j)^{2\delta + 2it_0} \frac{\Gamma\left( \frac k2 - \delta - it_0\right)}{\Gamma\left( \frac k2 + \delta + it_0 \right)} + \frac{(-1)^{k/2} (2\pi \tilde{\kappa}_j)^{\frac k2 +\delta + it_0} }{2\Gamma\left( \frac k2 + \delta + it_0\right)}\\
\nonumber && \times \sum_{\substack{(c,d)\in\mathbb{Z}^2 \\ (c,d)=1, ac>0}} c^{-\frac k2 - \delta - it_0} a^{-\frac k2 + \delta + it_0}\\
\nonumber  && \times
  \bigg( e^{2\pi i(n+\kappa_j)d/c} e^{\pi i(\frac k2 - \delta - it_0)}\chi^{-1}\left( \sm a&b\\c&d\esm\right) \rho^{-1}\left( \sm a&b\\c&d\esm\right)_{j,j}\ _1f_1(\frac k2 - \delta - it_0, k; -2\pi in/(ac))\\
 &&
 + e^{-2\pi i(n+\kappa_j)d/c} e^{-\pi i(\frac k2 - \delta - it_0)}\chi^{-1}\left( \sm -a&b\\c&-d\esm\right) \rho^{-1}\left( \sm -a&b\\c&-d\esm\right)_{j,j}\ _1f_1(\frac k2 - \delta - it_0, k; 2\pi in/(ac)) \bigg).
  \end{eqnarray}
  
 For $\mathrm{Re}(\beta) > \mathrm{Re}(\alpha)>0$, we have
 \[
 _1 f_1(\alpha, \beta;z) = \int_0^1 e^{zu} u^{\alpha-1}(1-u)^{\beta-\alpha-1} du.
 \]
   By \cite[13.21]{AS}, for $\mathrm{Re}(\alpha)>1, \mathrm{Re}(\beta-\alpha)>1$, and $|z|=1$,  we have
  \[
  |_1 f_1(\alpha, \beta; z)| \leq 1.
  \]
  If we take absolute values in (\ref{compare3}), then we have
\begin{equation} \label{compare4}
1 \leq (2\pi \tilde{\kappa}_j)^{2\delta} \frac{\left|\Gamma\left( \frac k2 - \delta - it_0\right)\right|}{\left|\Gamma\left( \frac k2 + \delta + it_0 \right)\right|} + \frac{(2\pi \tilde{\kappa}_j)^{\frac k2 +\delta}}{2\left|\Gamma\left( \frac k2 + \delta + it_0\right)\right|}  \sum_{\substack{(c,d)\in\mathbb{Z}^2 \\ (c,d)=1, ac>0}}  |c|^{-\frac k2-\delta} |a|^{-\frac k2 + \delta}(e^{\pi t_0}+e^{-\pi t_0}).
\end{equation}
 By \cite[6.147]{AS}, we have 
 \[
   \frac{\left|\Gamma\left( \frac k2 - \delta - it_0\right)\right|}{\left|\Gamma\left( \frac k2 + \delta + it_0 \right)\right|} \to \left(\frac k2\right)^{-2\delta - 2it_0}
   \]
   as $k\to\infty$. 
   Therefore, (\ref{compare4}) becomes $1\leq 0$ as $k\to\infty$, which is a contradiction.
   \end{proof}
   We now give a corollary that is a direct consequence of Theorem \ref{nonvanishing} which basically demonstrates the existence of a basis element of the space of vector-valued cusp forms whose $L$-function does not vanish.
   \begin{cor} \label{vvmfcor}
   Let $k\in\frac12\mathbb{Z}$ with $k>2$. 
  Let $\{f_{k,1}, \ldots, f_{k,d_k}\}$ be an orthogonal basis of $S_{k,\chi,\rho}$ with Fourier expansions
 \[
 f_{k,l}(z) = \sum_{j=1}^m \sum_{n+\kappa_j>0} b_{k,l,j}(n)e^{2\pi i(n+\kappa_j)z}\ (1\leq l\leq d_k).
 \]
  Let $t_0\in\mathbb{R}$ and $\epsilon>0$.
  \begin{enumerate}
     
  \item
  For any $k>C(t_0, \epsilon, j)$, any $1\leq j\leq m$, and any $s=\sigma + it_0$ with $$\frac{k-1}{2}<\sigma < \frac k2 - \epsilon,$$ there exists a basis element $f_{k,l}\in S_{k,\chi,\rho}$ such that
  \[
  <L^*(f_{k,l}, s), \mbf{e}_j> \neq 0\ \text{and}\ b_{k, l, j}(n_{j,0})\neq 0.
  \]

   \item There exists a constant $C(t_0, \epsilon)>0$ such that for any $k>C(t_0, \epsilon)$, and any $s=\sigma + it_0$ with $$\frac{k-1}{2}<\sigma < \frac k2 - \epsilon \ \ and \ \ \frac k2 + \epsilon < \sigma < \frac{k+1}2,$$ there exists a basis element $f_{k,l}\in S_{k,\chi,\rho}$ such that
  \[
  L(f_{k,l}, s)\neq 0.
  \]

  \end{enumerate}
  
  \end{cor}
  
\section{The Case of $\Gamma_0(N)$} \label{Section5}
  In what follows, we consider the case of a  scalar-valued modular form on the congruence subgroup $\Gamma_0(N)$. By using Theorem \ref{nonvanishing}, we can extend  Kohnen's result in \cite{Koh} to the case of $\Gamma_0(N)$. To illustrate,
  let $N$ be a positive integer and 
  let $\Gamma = \Gamma_0(N)$.
  Let $S_k(\Gamma)$ be the space of cusp forms of weight $k$ on $\Gamma$. 
  Let $\{\gamma_1, \ldots, \gamma_m\}$ be the set of representatives of $\Gamma \setminus \mathrm{SL}_2(\mathbb{Z})$ with $\gamma_1 = I$.
    For $f\in S_k(\Gamma)$, we define a vector-valued function $\tilde{f}:\mathbb{H}\to \mathbb{C}^m$ by $\tilde{f} = \sum_{j=1}^m f_j\mbf{e}_j$ and
    \[
    f_j = f|_k \gamma_j\ (1\leq j\leq m),
    \]
    where $(f|_k \sm a&b\\c&d\esm)(z) := (cz+d)^{-k} f(\gamma z)$.
    Then, $\tilde{f}$ is a vector-valued modular form of weight $k$ and the trivial multiplier system  with respect to $\rho$ on $\mathrm{SL}_2(\mathbb{Z})$, where $\rho$ is a certain $m$-dimensional unitary complex representation such that $\rho(\gamma)$ is a permutation matrix for each $\gamma\in \mathrm{SL}_2(\mathbb{Z})$ and is an identity matrix if $\gamma\in\Gamma$. 
    Then, the map $f\mapsto \tilde{f}$ induces an  isomorphism between $S_k(\Gamma)$ and $S_{k,\rho}$, where $S_{k,\rho}$ denotes the space of vector-valued cusp forms of weight $k$ and trivial multiplier system with respect to $\rho$ on $\mathrm{SL}_2(\mathbb{Z})$.
    
    Suppose that $f, g\in S_{k}(\Gamma)$. 
    Then, we have
    \begin{eqnarray*}
    (\tilde{f}, \tilde{g})& =& \int_{\mathcal{F}} <\tilde{f}, \tilde{g}> y^k \frac{dxdy}{y^2} \\
    &=& \sum_{j=1}^m \int_{\mathcal{F}}   (f|_k \gamma_j)(z) \overline{(g|_k \gamma)(z)}  y^k \frac{dxdy}{y^2}= (f,g)
    \end{eqnarray*}
    where $(f,g)$ denotes the Petersson inner product. 
    Therefore, if $f, g\in S_{k}(\Gamma)$ such that $f$ and $g$ are orthogonal, then $\tilde{f}$ and $\tilde{g}$ is also orthogonal.

    \begin{cor}
    Let $k$ be a positive even integer with $k>2$.
    Let $N$ be a positive integer and $\Gamma = \Gamma_0(N)$.
    Let $\{f_{k,1}, \ldots, f_{k, e_k}\}$ be an orthogonal basis of $S_k(\Gamma)$.
    %with Fourier expansions
% \[
% f_{k,l}(z) = \sum_{n>0} b_{k,l}(n)e^{2\pi inz}\ (1\leq l\leq e_k).
% \]    
   Let $t_0\in\mathbb{R}, \epsilon>0$. 
    Then, there exists a constant $C(t_0, \epsilon)>0$ such that for $k>C(t_0, \epsilon)$ there exists a basis element $f_{k,l}\in S_{k}(\Gamma)$ satisfying
  \[
  L(\widetilde{f_{k,l}}, s) \neq 0
  \]
 at any point $s=\sigma + it_0$ with $$\frac{k-1}{2}<\sigma < \frac k2 - \epsilon \ \ and \ \ \frac k2 + \epsilon < \sigma < \frac{k+1}2.$$
% where $L^*(f,s) = \frac{\Gamma(s)}{(2\pi)^s} L(f,s)$.

    \end{cor}

\section{The case of Jacobi forms} \label{Section6}
Let $k$ be a positive even integer and $m$ be a positive integer. 
Let $J_{k,m}$ be the space of Jacobi forms of weight $k$ and index $m$ on $\mathrm{SL}_2(\mathbb{Z})$.
From now, we use the notation $\tau = u+iv\in\HH$ and $z = x+iy\in\CC$.
We review basic notions of Jacobi forms (for more details, see \cite{EZ}).
Let $F$ be a complex-valued function on $\mathbb{H}\times \mathbb{C}$.
For $\gamma=\sm a&b\\c&d\esm \in\mathrm{SL}_2(\mathbb{Z}) , X = (\lambda, \mu)\in \mathbb{Z}^2$, we define
\[(F|_{k,m} \gamma)(\tau,z) := (c\tau+d)^{-k}e^{-2\pi im\frac{cz^2}{c\tau+d}}F(\gamma(\tau,z))\]
and
\[(F|_m X)(\tau,z) :=e^{2\pi i m (\lambda^2 \tau + 2\lambda z)}F(\tau,z+\lambda\tau+\mu),\]
where $\gamma(\tau,z) = (\frac{a\tau+b}{c\tau+d},\frac{z}{c\tau+d})$.

With these notations, we introduce the definition of a Jacobi form.

\begin{dfn}
A {\it{Jacobi form}} of weight $k$ and index $m$ on $\mathrm{SL}_2(\mathbb{Z})$  is a holomorphic function $F$ on $\mathbb{H}\times\mathbb{C}$ satisfying
\begin{enumerate}
\item[(1)] $F|_{k,m} \gamma =F$ for every $\gamma\in\mathrm{SL}_2(\mathbb{Z})$,
\item[(2)] $F|_m X = F$ for every $X\in \mathbb{Z}^2$,
\item[(3)] $F$ has the Fourier expansion of the form
\begin{equation} \label{Jacobifourier}
F(\tau,z) =
\sum_{l, r\in\mathbb{Z}\atop 4ml - r^2 \geq0}a(l,r)e^{2\pi il\tau}e^{2\pi irz}.
\end{equation}
\end{enumerate}
\end{dfn}

We denote by $J_{k,m}$ the vector space of all Jacobi forms of weight $k$ and index $m$ on $\mathrm{SL}_2(\mathbb{Z})$. 
If a Jacobi form satisfies the  condition $a(l,r)\neq 0$ only if $4ml - r^2>0$, then it is called a Jacobi cusp form. 
We denote by $S_{k,m}$ the vector space of all Jacobi cusp forms of weight $k$ and  index $m$ on $\mathrm{SL}_2(\mathbb{Z})$.

For $1\leq j\leq 2m$, we consider the theta series
\[
\theta_{m, j}(\tau,z) := \sum_{r\in\mathbb{Z} \atop r\equiv j \pmod{2m}} e^{2\pi ir^2\tau/(4m)} e^{2\pi irz}.
\]
Suppose that  $F(\tau,z)$ is a  holomorphic function of $z$ and satisfy
\begin{equation*} \label{elliptic}
F|_m X = F\ \text{for every}\ X\in \mathbb{Z}^2.
\end{equation*}
Then we have
\begin{equation} \label{thetaexpansion}
F(\tau,z) = \sum_{1\leq j\leq 2m}F_j(\tau)\theta_{m,j}(\tau,z)
\end{equation}
with uniquely determined holomorphic functions $F_a:\mathbb{H}\to\mathbb{C}$.
Furthermore, if $F$ is a Jacobi form in $J_{k,m}$ with the Fourier expansion (\ref{Jacobifourier}), then
functions in $\{F_j|\ 1\leq j\leq 2m\}$ have the Fourier expansions 
\[
F_j(\tau) =\sum_{n\geq0\atop n+j^2 \equiv 0 \pmod{4m}} a\left(\frac{n+j^2}{4m}, j\right)e^{2\pi in\tau/(4m)}.
\]

In \cite{CL0}, it is proved that the Petersson inner product of skew-holomorphic Jacobi cusp forms can be expressed as the sum of partial $L$-values of skew-holomorphic Jacobi cusp forms. 
Similarly, for a Jacobi cusp form $F\in J_{k,m}$ with its Fourier expansion (\ref{Jacobifourier}), we define partial $L$-functions of $F$ by
\[
L(F,j,s) := \sum_{n\in\mathbb{Z} \atop  n+j^2 \equiv 0 \pmod{4m}} \frac{a\left(\frac{n+j^2}{4m}, j\right)} {\left(\frac{n}{4m}\right)^{s}}
\]
for $1\leq j\leq 2m$.

We write $\Mp_2(\mathbb{R})$ for the metaplectic group.
The elements of $\Mp_2(\mathbb{R})$ are pairs $(\gamma,\phi(\tau))$, 
where $\gamma = \sm a&b\\c&d\esm\in\SL_2(\mathbb{R})$, and $\phi$ denotes a holomorphic function on $\HH$ with $\phi(\tau)^2= c\tau+d$.
The product of $(\gamma_1,\phi_1(\tau)), (\gamma_2, \phi_2(\tau))\in\Mp_2(\mathbb{R})$ is given by
\[(\gamma_1,\phi_1(\tau)) (\gamma_2, \phi_2(\tau)) = (\gamma_1\gamma_2, \phi_1(\gamma_2\tau)\phi_2(\tau)).\]
The map
\[\sm a&b\\c&d\esm  \mapsto \widetilde{\sm a&b\\c&d\esm} = (\sm a&b\\c&d\esm, \sqrt{c\tau+d})\]
defines a locally isomorphic embedding of $\SL_2(\mathbb{R})$ into $\Mp_2(\mathbb{R})$. Let $\Mp_2(\mathbb{Z})$ be the inverse image of $\SL_2(\mathbb{Z})$ under the covering map $\Mp_2(\mathbb{R})\to\SL_2(\mathbb{R})$. It is well known that $\Mp_2(\mathbb{Z})$ is generated by $\widetilde{T}$ and $\widetilde{S}$. 

 We define a  $2m$-dimensional unitary complex representation $\widetilde{\rho}_m$ of $\Mp_2(\mathbb{Z})$ by 
 \[
 \widetilde{\rho}_m(\widetilde{T})  \mbf{e}_j = e^{-2\pi ij^2/(4m)}\mbf{e}_j
 \]
 and
 \[
  \widetilde{\rho}_m(\widetilde{S})  \mbf{e}_j = \frac{i^{\frac 12}}{\sqrt{2m}}\sum_{j'=1}^{2m} e^{2\pi i jj'/(2m)}\mbf{e}_{j'},
 \]
 Let $\chi$ be a multiplier system of weight $\frac12$ on $\mathrm{SL}_2(\mathbb{Z})$.
 We define a map $\rho_m : \SL_2(\mathbb{Z}) \to \GL_{2m}(\mathbb{C})$ by
\[\rho_m(\gamma) = \chi(\gamma)\widetilde{\rho}_m(\widetilde{\gamma})\]
for $\gamma\in\SL_2(\mathbb{Z})$.
The map $\rho_m$ gives a $2m$-dimensional unitary representation of $\SL_2(\mathbb{Z})$.

Let $\{\mbf{e}_1, \ldots, \mbf{e}_{2m}\}$ denote the standard basis of $\mathbb{C}^{2m}$. 
For $F\in S_{k,m}$, we define a vector-valued function $\tilde{F}:\mathbb{H}\to \mathbb{C}^{2m}$ by $\tilde{F} = \sum_{j=1}^{2m} F_j\mbf{e}_j$, where $F_j$ is defined by the theta expansion in (\ref{thetaexpansion}).
Then, the map $F\mapsto  \tilde{F}$ induces an  isomorphism between $S_{k,m}$ and  $S_{k-\frac12, \bar{\chi}, \rho_m}$ (for more details, see \cite[Section 5]{EZ} and \cite[Section 3.1]{CL}). 

 Suppose that $F,G\in S_{k,m}$.
 The Petersson inner product of $F$ and $G$ by 
 \[
 (F,G) := \int_{\mathrm{SL}_2(\mathbb{Z})^J \setminus \mathbb{H}\times \mathbb{C}} v^k e^{-4\pi my^2/v}F(\tau,z) \overline{G(\tau,z)} \frac{dxdydudv}{v^3},
 \]
 where $\mathrm{SL}_2(\mathbb{Z})^J = \mathrm{SL}_2(\mathbb{Z}) \ltimes \mathbb{Z}^2$.
 Then, by Theorem 5.3 in \cite{EZ}, we have
  \[
  (F, G) = \frac{1}{\sqrt{2m}} (\tilde{F}, \tilde{G}).
  \]

  Note that $\rho_m(-I)$ is not equal to the identity matrix in $\mathrm{GL}_{2m}(\mathbb{C})$. 
  Instead, we have
  \[
  \rho_m(-I) \mbf{e}_j = i \mbf{e}_{2m-j}.
  \]
Then, the corresponding kernel function $R_{k,s,l}$ has the Fourier expansion
\[
R_{k,s,l}(\tau) = \sum_{j=1}^{2m} \sum_{n+\kappa_j>0} r_{k,s,l,j}(n)e^{2\pi i(n+\kappa_j)\tau},
\]
where $r_{k,s,l,j}(n)$ is given by
\begin{eqnarray*}
r_{k,s,l,j}(n) &=& \frac12 \delta_{l,j}  (2\pi)^s \Gamma(k-s)(n+\kappa_i)^{s-1}  + \frac i2 \delta_{2m-l,j}  (2\pi)^s \Gamma(k-s)(n+\kappa_{2m-l})^{s-1}\\
&&+ \frac12 \chi^{-1}(S)   \rho^{-1}\left( S \right)_{j,l} (-1)^{k/2}(2\pi)^{k-s}\Gamma(s) (n+\kappa_j)^{k-s-1}\\
&&+ \frac i2 \chi^{-1}(S)   \rho^{-1}\left( S \right)_{j,2m-l} (-1)^{k/2}(2\pi)^{k-s}\Gamma(s) (n+\kappa_j)^{k-s-1}
\\
&&+\frac{ (-1)^{k/2}}{2} (2\pi)^k (n+\kappa_j)^{k-1} \frac{\Gamma(s)\Gamma(k-s)}{\Gamma(k)}\sum_{\substack{(c,d)\in\mathbb{Z}^2 \\ (c,d)=1, ac>0}}    c^{-k} \left( \frac ca \right)^{s}\\
&&\quad\times \bigg(e^{2\pi i(n+\kappa_j)d/c} e^{\pi is}\chi^{-1}\left( \sm a&b\\c&d\esm\right) \rho^{-1}\left( \sm a&b\\c&d\esm\right)_{j,l}\ _1F_1(s, k; -2\pi in/(ac))\\
&&\quad \quad + e^{-2\pi i(n+\kappa_j)d/c} e^{-\pi is} \chi^{-1}\left( \sm -a&b\\c&-d\esm\right) \rho^{-1} \left( \sm -a&b\\c&-d\esm\right)_{j,l}\ _1F_1(s, k; 2\pi in/(ac)) \bigg).
\end{eqnarray*}
By the similar argument, we prove the same result as in Corollary \ref{vvmfcor} for the representation $\rho_m$.
From this, we have the following corollary. 

 \begin{cor}
   Let $k$ be a positive even integer with $k>2$. 
  Let $\{F_{k,m,1}, \ldots, F_{k,m,d}\}$ be an orthogonal basis of $S_{k,m}$.
  Let $t_0\in\mathbb{R}$ and $\epsilon>0$.
  \begin{enumerate}
     
  \item
  For any $k>C(t_0, \epsilon, j)$, any $1\leq j\leq 2m$, and any $s=\sigma + it_0$ with $$\frac{2k-3}{4}<\sigma < \frac {2k-1}{4} - \epsilon,$$ there exists a basis element $F_{k,m,l}\in S_{k,m}$ such that
  \[
   L(F_{k,m,l},j,s)\neq 0.
  \]

   \item There exists a constant $C(t_0, \epsilon)>0$ such that for any $k>C(t_0, \epsilon)$, and any $s=\sigma + it_0$ with $$\frac{2k-3}{4}<\sigma < \frac {2k-1}{4} - \epsilon \ \ and \ \ \frac {2k-1}{4} + \epsilon < \sigma < \frac{2k+1}4,$$ there exist a basis element $F_{k,m,l}\in S_{k,m}$ and $j\in\{1,\ldots, 2m\}$ such that
  \[
  L(F_{k,m,l},j,s) \neq 0.
  \]

  \end{enumerate}
  
  \end{cor}

\newpage
  
    %%%%%%%%%%%%%%
    % References
    %%%%%%%%%%%%%%%%

    \end{document}